%% file: R-F-K.tex
\documentclass[11pt,a4paper,reqno]{amsart}

\input{preamble}

\makeatother

\date{}
\begin{document}
\setstretch{1.1} 

\title[Reverse Faber-Krahn inequalities for Zaremba problems]{Reverse Faber-Krahn inequalities for Zaremba problems} 

\author[T. V. Anoop and Mrityunjoy Ghosh]{T. V. Anoop$^{1,*}$ and Mrityunjoy Ghosh$^{1}$}
\blfootnote{$^{*}$Corresponding author}
\address{$^{1}$Department of Mathematics, Indian Institute of Technology Madras, Chennai 600036, India}

\email{anoop@iitm.ac.in}

\email{ghoshmrityunjoy22@gmail.com}



\subjclass[2020]{35P15, 35P30, 49R05, 49Q10}
\keywords{Zaremba problems, $p$-Laplacian, Reverse Faber-Krahn inequality, Quermassintegrals, Steiner Formula, Nagy's inequality, Method of interior parallel sets.}

\begin{abstract}
    Let $\Omega$ be a multiply-connected domain in $\mathbb{R}^n$ ($n\geq 2$) of the form $\Omega=\Omega_{\text{out}}\setminus \overline{\Omega_{\text{in}}}.$ Set $\Omega_D$ to be  either $\Omega_{\text{out}}$ or $\Omega_{\text{in}}$.  For $p\in (1,\infty),$ and $q\in [1,p],$ let $\tau_{1,q}(\Omega)$ be the first eigenvalue of
    \begin{equation*}
        -\Delta_p u =\tau \left(\int_{\Omega}|u|^q \text{d}x \right)^{\frac{p-q}{q}} |u|^{q-2}u\;\text{in} \;\Omega,\;
	u =0\;\text{on}\;\partial\Omega_D,
\frac{\partial u}{\partial \eta}=0\;\text{on}\; \partial \Omega\setminus \partial \Omega_D.
    \end{equation*}
    Under the assumption that $\Omega_D$ is convex, we establish the following reverse Faber-Krahn inequality
    $$\tau_{1,q}(\Omega)\leq \tau_{1,q}({\Omega}^\bigstar),$$
 where ${\Omega}^\bigstar=B_R\setminus \overline{B_r}$ is a concentric annular region in $\mathbb{R}^n$ having the same Lebesgue measure as $\Omega$ and such that \begin{enumerate}[(i)]
     \item (when $\Omega_D=\Omega_{\text{out}}$) $W_1(\Omega_D)= \omega_n R^{n-1}$, and  $(\Omega^\bigstar)_D=B_R$,
     \item (when $\Omega_D=\Omega_{\text{in}}$) $W_{n-1}(\Omega_D)=\omega_nr$, and 
     $(\Omega^\bigstar)_D=B_r$.
      \end{enumerate} 
    Here $W_{i}(\Omega_D)$ is the $i^{\text{th}}$ {\it quermassintegral} of $\Omega_D.$  We also establish Sz. Nagy's type inequalities for parallel sets of a convex domain in $\mathbb{R}^n$ ($n\geq 3$) for our proof.
\end{abstract}


\maketitle
\input{Intro}

\input{prelim}

\input{nagy}

\input{outerD}
\input{innerD}
\input{remarks}

\bibliographystyle{abbrvurl}
\bibliography{refconvex}

\end{document}

%% file: preamble.tex
\usepackage{setspace}
\usepackage{footnote}
\usepackage{graphicx}
\usepackage{cmap,mathtools}
\usepackage[T1]{fontenc}
\usepackage[utf8]{inputenc}
\usepackage[english]{babel}
\usepackage{amsfonts,amssymb,amsthm,amsmath,enumerate,bm, cite,mathrsfs, soul}
\usepackage{url}
\usepackage[shortlabels]{enumitem}
\usepackage{float}
\usepackage{secdot}

\allowdisplaybreaks

\usepackage{graphicx}
\usepackage{epstopdf}
\usepackage{a4wide}
\usepackage{xparse}
\usepackage[dvipsnames]{xcolor}
\usepackage[normalem]{ulem}

\usepackage[colorlinks=true,linktocpage,pdfpagelabels,bookmarksnumbered,bookmarksopen]{hyperref}
\definecolor{ForestGreen}{rgb}{0.1,0.7,0}
\definecolor{EgyptBlue}{rgb}{0.063,0.1,0.6}
\hypersetup{
	colorlinks=true,
	linkcolor=blue,         
	citecolor=ForestGreen,
	urlcolor=ForestGreen
}

\usepackage[hyperpageref]{backref} 

\usepackage[margin=1.1in]{geometry}

\newcommand\blfootnote[1]{%
  \begingroup
  \renewcommand\thefootnote{}\footnote{#1}%
  \addtocounter{footnote}{-1}%
  \endgroup
}

\newtheorem{theorem}{Theorem}
\newtheorem{proposition}[theorem]{Proposition}
\newtheorem{lemma}[theorem]{Lemma}
\newtheorem{corollary}[theorem]{Corollary}
\theoremstyle{definition}

\newtheorem{remark}[theorem]{Remark}

\let\OLDthebibliography\thebibliography
\renewcommand\thebibliography[1]{
	\OLDthebibliography{#1}
	\setlength{\parskip}{1pt}
	\setlength{\itemsep}{1pt plus 0.3ex}
}

\numberwithin{theorem}{section}

\numberwithin{theorem}{section}

\DeclarePairedDelimiter\norm{\lVert}{\rVert}%
\makeatletter
\let\oldnorm\norm
\def\norm{\@ifstar{\oldnorm}{\oldnorm*}}

\newcommand{\al} {\alpha}

\newcommand{\pa} {\partial}
\newcommand{\be} {\beta}
\newcommand{\de} {\delta}
\newcommand{\De} {\Delta}

\newcommand{\ga} {\gamma}
\newcommand{\Ga} {\Gamma}
\newcommand{\ta}{\tau}
\newcommand{\om} {\omega}
\newcommand{\Om} {\Omega}

\newcommand{\la} {\lambda}

\newcommand{\si} {\sigma}

\newcommand{\no} {\nonumber}
\newcommand{\noi} {\noindent}

\newcommand{\cdast}{\circledast}

\newcommand{\RN}{{\mathbb{R}^n}}

\newcommand{\da}{{\,\rm d}\al}
\newcommand{\dd}{{\,\rm d}\de}
\newcommand{\dsi}{{\,\rm d}\si}

\newcommand\restr[2]{{
  \left.\kern-\nulldelimiterspace 
  #1 
  \right|_{#2} 
  }}

\def\C{{\mathcal C}}

\def\R{{\mathbb R}}

\def\({{\Big(}}
\def\){{\Big)}}

\def\c1{{\C_c^1}}

\def\d{{\rm d}}

\def\dt{{\rm d}t}

\def\dx{{\rm d}x}

\def\drho{{\rm d}\rho}

\def\Omo{{\Omega_{\text{out}}}}
\def\Omi{{\Omega_{\text{in}}}}

\setstcolor{red}

\usepackage[foot]{amsaddr}
\usepackage[pagewise]{lineno}

\makeatletter
\@namedef{subjclassname@2020}{\textup{2020} Mathematics Subject Classification}
\makeatother


%% file: Intro.tex
\section{Introduction and statements of the main results}\label{intro}
In the book \textit{The Theory of Sound} \cite{Rayleigh}, Lord Rayleigh conjectured that - ``\textit{among all the planar domains with the fixed area, disk minimizes the first Dirichlet eigenvalue of the Laplacian}.'' To state this result mathematically, consider the first eigenvalue  $\la_1(\Om)$ of the following eigenvalue problem on a bounded domain $\Om\subset \mathbb{R}^n$:
\begin{equation}\label{Dirichlet}
	-\Delta u =\la u \quad\text{in} \quad\Omega,\;\;
	u =0 \quad\text{on}\quad \pa\Om.
\end{equation}
Let $\Om^*$  be the open ball centered at the origin with the same volume as $\Om$.
Then the Rayleigh's conjecture reads as: for $n=2,$
\begin{equation*}
    \la_1(\Om^*)\leq \la_1(\Om).
\end{equation*}
 Rayleigh's conjecture  was proved independently by Faber \cite{Faber} and Krahn \cite{Krahn1925}. Later Krahn extended the result for $n> 2$ in \cite{Krahn1926}. The proof given by Faber is based on discretization  and  approximation, whereas Krahn's proof \cite{Krahn1925,Krahn1926} is based on the classical  isoperimetric inequality and the Coarea formula. In \cite{Daners2007}, it is shown that the equality in Rayleigh-Faber-Krahn inequality  holds only if $\Om$ is a ball up to a set of  capacity zero.
 
 Now let us consider the first non-trivial Neumann eigenvalue $\mu_2(\Om)$ of the following Neumann eigenvalue problem on a bounded Lipschitz domain $\Om\subset \mathbb{R}^n$:
\begin{equation}\label{Neumann}
	-\Delta u =\mu u \quad\text{in} \quad\Omega,\;\;
	\frac{\partial u}{\partial \eta} =0 \quad\text{on}\quad \pa\Om.
\end{equation}
where $\eta$ is the unit outward normal to  $\partial \Om$. For $n=2$, Kornhauser and Stakgold \cite{Stakgold} have shown that 
\begin{equation}\label{Szego}
    \mu_2(\Om)\leq \mu_2(\Om^*),
\end{equation}
provided $\Om$ is obtained by a small area-preserving perturbation of  a disk. Further, they have established that a maximizing domain (if exists) for $\mu_2$, in the class of simply-connected planar domains with the fixed area, must be  a disk. Later, Szeg\"{o} \cite{Szego} proved \eqref{Szego} for  a simply-connected planar  domain $\Om$ which is bounded by an analytic curve. In 1956, Weinberger \cite{Weinberger1956} extended \eqref{Szego} for general bounded Lipschitz domain in $\R^n$ without the simply connectedness assumption. The inequality \eqref{Szego} is known as the Szeg\"{o}-Weinberger inequality in the literature. The equality in Szeg\"{o}-Weinberger inequality holds if and only if $\Om$ is a ball up to a set of Lebesgue measure zero; see, for instance, \cite[Theorem 7.1.1]{Henrot2006}.

For $1 < p < \infty,$ we study the similar inequalities  for the first eigenvalue of the Zaremba problems (mixed boundary conditions) for the $p$-Laplace operator $\De_p,$ defined by $\Delta_p u\coloneqq \text{div}(|\nabla u|^{p-2}\nabla u)$, on multiply-connected domains. More precisely,  we consider domains of the following form: 
\begin{equation}\label{Domain}
 \left\{\begin{aligned}
      \Om&=\Omo\setminus \overline{\Omi},\; \Om\;\text{is Lipschitz}\\ &\text{and}\;
  \Omo, \Omi\;\text{are open sets in}\;\R^n\;
  \text{ such that}\;\overline{\Omi}\subset\Omo.
\end{aligned}\right.
\end{equation}
 Let $\Om_D$ be either $\Omo$ or $\Omi$ and $\Ga_D:=\pa \Om_D$. Now for $1\leq q\leq p$, consider the following Zaremba eigenvalue problem for the $p$-Laplacian 
on $\Om$:
\begin{equation}\tag{P}\label{Problem}
	\left. \begin{aligned}
	-\Delta_p u &=\ta \left(\int_{\Om}|u|^q \dx \right)^{\frac{p-q}{q}} |u|^{q-2}u \quad\text{in} \quad\Omega,\\
	u &=0 \qquad\qquad \quad\; \qquad \qquad \qquad\;\text{on}\quad \Ga_D,\\
\frac{\partial u}{\partial \eta} &=0 \qquad\qquad \quad\; \qquad \qquad \qquad\;\text{on} \quad \pa \Om\setminus\Ga_D,
	\end{aligned}\right\}
\end{equation}
where $ \ta\in\R $ and $\eta$ is the unit outward normal to the boundary of $\Om$. For $q=p$, \eqref{Problem} coincides with the eigenvalue  problem  for the $p$-Laplace operator. Indeed, \eqref{Problem} admits a least positive eigenvalue $\ta_{1,q}(\Om)$ (see Proposition \ref{Existence}) and it has the following variational characterization:
\begin{equation}\label{variational_lambda}
	\ta_{1,q}(\Omega)=\displaystyle\inf\Big\{\mathcal{R}_q(u): u\in W^{1,p}_{\Ga_D}(\Omega)\setminus\{0\} \Big\},
\end{equation}
where $\mathcal{R}_q(u):=\frac{\int_{\Omega}|\nabla u|^p\dx}{\left(\int_{\Omega}|u|^q\dx\right)^\frac{p}{q}}$ and $W^{1,p}_{\Ga_D}(\Om)$ is the space of all functions in $W^{1,p}(\Om)$ that vanishes on $\Ga_D$. In fact, $\frac{1}{\ta_{1,q}(\Om)}$ is the best constant of the Sobolev embedding $W^{1,p}_{\Ga_D}(\Om)\hookrightarrow L^q(\Om),$ for $q\in [1,p].$ 

The shape optimization problems of such nonlinear eigenvalue are considered by Bucur and Giacomini \cite{Bucur2015Giacomini} for the first Robin eigenvalue of the Laplacian. They have established a family of Faber-Krahn type inequalities (for $1\leq q\leq \frac{2n}{n-2}$) for the first eigenvalue of the Robin Laplacian. In \cite{Bobkov2020Kolonitskii}, Bobkov and Kolonitskii studied a monotonicity result (with respect to domain perturbation) for the first Dirichlet eigenvalue of the $p$-Laplacian with suitable nonlinearity; see \cite{anoop2018,AnisaMrityunjoy} also for a similar result when $q=p$. The symmetries of the minimizers of $\mathcal{R}_q$ subject to different boundary conditions can be found in \cite{Nazarov2000,Nazarov2001, Kawohl}. We also refer to the monographs \cite{Henrot2006,henrot2021} for further readings in this direction.

In this article, we consider the following two cases:
\begin{center}
\begin{minipage}{.5\textwidth}
    \begin{enumerate}[(i)]
\item \textit{Outer Dirichlet problem}: $ \Om_D=\Omo$.
    \item \textit{Inner Dirichlet problem}: $ \Om_D=\Omi$. 
\end{enumerate}
\end{minipage}
    \end{center}
 For a Lebesgue measurable set  $A$ in $\R^n,$ $|A|$ denotes the Lebesgue measure of $A$ and $P(A)$ denotes the perimeter ($(n-1)$-dimensional Hausdorff measure of $\pa A$) of $A$. For $a>0$, $B_a$ denotes the open ball of radius $a$ centered at the origin. For each of the outer and inner Dirichlet problems, we choose $0<r<R$ (whence two concentric annular regions) as follows:
\begin{enumerate}[label=(\textbf{A\arabic*})]
    \item \label{AO}\textit{For outer Dirichlet problem}:
   $$A_O(\Om)=B_R\setminus \overline{B_r} \text{ such that } |\Om|=|A_O(\Om)| \text{ and }  P(\Om_D)=P(B_R),$$
    
    \item \label{AI}\textit{For inner Dirichlet problem}:
    $$A_I(\Om)=B_R\setminus \overline{B_r} \text{ such that } |\Om|=|A_I(\Om)| \text{ and }  P(\Om_D)=P(B_r).$$
\end{enumerate}

\subsection{Outer Dirichlet problem \normalfont{($ \Om_D=\Omo$)}}
 In \cite{Payne}, for $n=2$, $p=q=2$, Payne and Weinberger established that
\begin{equation}\label{Payne_rfk}
    \ta_{1,2}(\Om)\leq \ta_{1,2}(A_O(\Om)).
\end{equation}
  In particular, by setting $\Omi=\emptyset,$ it is easy to see that over the class of domains with  fixed area and perimeter, $\la_1(\Om)$ (first Dirichlet eigenvalue) is bounded above by $\tau_{1,2}(A_O(\Om)).$ This result sharpens the various upper bounds for $\la_1(\Om)$ obtained earlier in \cite{Polya,Makai}. In \cite{Hersch}, Hersch provided a different proof of \eqref{Payne_rfk}. Both Hersch and Payne-Weinberger used Sz. Nagy's inequality \cite{Nagy} for inner parallels sets. Now we recall the inner parallel set of a domain and the Sz. Nagy's inequality for the inner parallel set. For a non-empty set $A\subset \R^n$, we define $d(x, A)=\inf\{d(x,y): y\in A\},$ where $d$ is the euclidean distance function.
 \subsubsection{Inner parallel set:}\label{Inner_parallel} Let $\Om\subset \mathbb{R}^n$ be a bounded domain and $\Ga\subset \pa\Om$. For $\de>0$, consider the set $$\Om_{-\de}:=\{x\in \Om:d(x,\Ga)\geq\de\}.$$ The set $\pa\Om_{-\de}$ is known as the \textit{inner parallel set} with respect to $\Ga$ at a distance $\de$. For $n=2$, $\Om$ simply connected and $\Ga=\pa \Om$, Sz. Nagy \cite{Nagy} proved that
  \begin{equation*}
    P(\Om_{-\de})\leq  P( \Om)-2\pi\de, \quad \forall\, \de>0.
\end{equation*}
  For $\Om$, we define $\Om^\#$ in the following way:
  \begin{equation}\label{Ball_perimeter}
 \begin{aligned}
      \Om^\#\;\text{is the open ball centered at the origin such that}\;P(\Om^\#)=P(\Om).
\end{aligned}
\end{equation}
  For $n=2$, it is easy to observe that $P(\Om^\#_{-\de})=P(\Om)-2\pi\de$ and hence Sz. Nagy's inequality can be restated as
\begin{equation}\label{Nagy_in}
    P(\Om_{-\de})\leq P(\Om^\#_{-\de}).
\end{equation}
 In \cite{Anoop2020}, for $n\geq 3$ and $\Om$ as given in \eqref{Domain}, Anoop and Ashok observed that the equation \eqref{Nagy_in} holds for $\Ga=\Ga_D$ when $\Om_D$ is a ball in $\mathbb{R}^n$. Using \eqref{Nagy_in}, for $q=p$, they extended reverse Faber-Krahn inequality \eqref{Payne_rfk} for the first eigenvalue $\ta_{1,q}(\Om)$ of \eqref{Problem} to higher dimensions under the assumption that  $\Om_D$ is a ball; cf \cite[Theorem 1.1]{Anoop2020}. They named \eqref{Payne_rfk} as the \textit{reverse Faber-Krahn inequality}. In \cite{Paoli}, the authors recently extended this result (for $q=p$) for $\Om$ with $\Om_D$ as a convex domain. Their proof is based on constructing a web function using the first eigenfunction of \eqref{Problem} on  $A_O(\Om)$. In this article, we establish this result for $\ta_{1,q}$ using a different method explicitly based on the Sz. Nagy's type inequality for inner parallel sets in higher dimensions. Furthermore, we prove that the concentric annulus is the unique maximizer (up to a translation) for $\ta_{1,q}$ in the admissible class of Lipschitz domains. First, we establish \eqref{Nagy_in} for convex domains in higher dimensions using the Alexandrov-Fenchel inequality (Section \ref{ReverseNagy}).
 
 \begin{proposition}\label{Nagy_in_extension}
Let $\Om\subset\mathbb{R}^n$ be a bounded, convex domain and $r_\Om$ be the inradius of $\Om$. Let $\Om^\#$ be as defined in \eqref{Ball_perimeter}. Then
\begin{enumerate}[(i)]
    \item $P(\Om_{-\de})\leq P(\Om^\#_{-\de}),\;\text{for all}\;\de\in (0,r_\Om),$
    
    \item for $n\geq 3,$ the equality holds in the above inequality for some $\de$ if and only if $\Om$ is a ball.
\end{enumerate}
\end{proposition}
Next, as an application of Proposition \ref{Nagy_in_extension}, we extend \eqref{Payne_rfk} for the  nonlinear eigenvalues $\ta_{1,q}(\Om)$ of \eqref{Problem} provided that $\Om_D$ is convex.
\begin{theorem}
\label{Theorem_Out}
Let $\Om$ and $A_O(\Om)$ be as given in \eqref{Domain} and \ref{AO}, respectively, with $\Om_D=\normalfont{\Omo}$. For $q\in[1,p]$, let $\ta_{1,q}(\Om)$ be the first eigenvalue of \eqref{Problem} on $\Om$. If $\Om_D$ is convex, then
\begin{equation*}\label{RFK_out}
    \ta_{1,q}(\Om)\leq \ta_{1,q}(A_O(\Om)).
\end{equation*}
Furthermore, for $n\geq 3,$ the equality holds if and only if $\Om=A_O(\Om)$ (up to a translation).
\end{theorem}
\subsection{Inner Dirichlet problem \normalfont{($\Om_D=\Omi$)}}
In \cite{Hersch}, Hersch  consider \eqref{Problem}  for the case $\Om_D=\Omi.$ For $n=2$ and $p=q=2$, Hersch established the following reverse Faber-Krahn inequality: 
 \begin{equation}\label{Hersch_rfk} 
     \ta_{1,2}(\Om)\leq \ta_{1,2}(A_I(\Om)),
 \end{equation}  
where $A_I(\Om)$ is as defined in \ref{AI}. The key step in proving \eqref{Hersch_rfk} is Sz. Nagy's inequality \cite{Nagy} for outer parallels for a bounded planar domain $\Om$. Let us define the outer parallel sets and describe the Sz. Nagy's inequality for them.
 \subsubsection{Outer parallel set:}\label{Outer_parallel} For a bounded domain $\Om\subset \mathbb{R}^n$, the \textit{outer parallel body} $\Om_{\de}$ of $\Om$  with respect to $\Ga\subset \pa\Om$  at a distance $\de>0$ is defined by
 \begin{equation}\label{Outer_p}
     \Om_{\de}:=\{ x\in \Om^c: d(x,\Ga)\geq \de \}^c.
 \end{equation} The boundary $\pa\Om_{\de}$ is called as the \textit{outer parallel set} at a distance $\de$. For $n=2$, $\Om$ simply connected and $\Ga=\pa\Om$, Sz. Nagy's inequality \cite{Nagy} for outer parallel sets states that
\begin{equation}\label{Nagy_out}
    P(\Om_{\de})\leq P(\Om^\#_{\de}),
\end{equation}
where $\Om^\#$ is as given in \eqref{Ball_perimeter}. 
\begin{remark}
Notice that, for a convex planar domain, Steiner formula gives the equality in  \eqref{Nagy_out}; see \cite{Steiner} or \cite[Theorem 10.1]{Gray}.
\end{remark}
In \cite{Anoop2020}, for $n\geq 3$ and $\Om$ as given in \eqref{Domain}, Anoop and Ashok observed that the inequality \eqref{Nagy_out} holds for $\Ga=\Ga_D$ when $\Om_D$ is a ball in $\mathbb{R}^n$. As a consequence, they obtained \eqref{Hersch_rfk} for the first eigenvalue $\ta_{1,q}(\Om)$ of \eqref{Problem} with $q=p$ in higher dimensions under the assumption that $\Om_D$ is a ball; cf. \cite[Theorem 1.2]{Anoop2020}. We observed that the inequality \eqref{Nagy_out} with $\Ga=\pa\Om$ fails in higher dimensions, even  for a general convex domain. More precisely, we establish the following reverse type of Sz. Nagy's inequality for convex sets in higher dimensions.
\begin{theorem}\label{nag}
 Let $\Om\subset\mathbb{R}^n\;(n\geq 3)$ be a bounded, convex domain and $\de>0$. Let $\Om^\#$ be as defined in \eqref{Ball_perimeter}. Then
 \begin{equation*}\label{Nagy_reverse}
     P(\Om_{\de})\geq P(\Om^\#_{\de}).
 \end{equation*} 
  Furthermore, equality holds for some $\de$ if and only if $\Om$ is a ball.
 \end{theorem}
 
 Thus \eqref{Nagy_out} fails for $\Om^\#$,  which is chosen with the perimeter constraint.  In view of the above theorem, to extend \eqref{Nagy_out} to higher dimensions, we need to come up with a
constraint that gives a ball $B$ for which the following inequality holds:
\begin{equation}\label{Nagy_O}
    P(\Om_\de)\leq P(B_\de),\;\forall \;\de>0.
\end{equation} 
For this, we consider the Steiner formula available for convex domains in higher dimensions.


\subsubsection{Steiner Formula for a convex domain}
\label{Steiner_section}
Let $\Om\subset \mathbb{R}^n$ be a non-empty bounded convex domain and $\delta>0$. Then for $\Om_\de$ as given in \eqref{Outer_p} with $\Ga=\pa \Om$, the \textit{Steiner formula} (cf. \cite[Chapter 4]{Schneider}) provides an expression for the perimeter of $\Om_{\de}$ as a polynomial in $\delta$:
\begin{equation}
\label{Steiner_formula}
	P(\Om_{\de})= n \sum_{i=0}^{n-1} \binom {n-1}{i} W_{i+1}(\Om)\delta^i,
\end{equation}
where the co-efficients $W_1(\Om),\dots,W_n(\Om)$ are called \textit{Quermassintegrals} or
 \textit{Minkowski  functionals} of $\Om$, which are special cases of mixed volumes (cf. \cite[Section 19.1]{Burago} or \cite[Section 4.2]{Schneider}). In particular, 
 $$ W_0(\Om)=|\Om|,\; W_1= \frac{P(\Om)}{n}\;, \;W_n(\Om)= \omega_n,$$
where $\om_n$ is the volume of the unit ball in $\RN$. For an open ball $B\subset \R^n\;(n\geq 2)$ of radius $R$, the quermassintegrals can be computed using \cite[(4.2.28)]{Schneider} as below:
\begin{equation}
\label{Quermassintegral_ball}
	W_j(B)=\omega_n R^{n-j}, \;\; \text{for}\;\; 0\leq j\leq n-1.
\end{equation}

Now, it is clear that if we choose  $B$ such that $W_i(\Om)\le W_i(B)$ for $i=1,2,\ldots, n-1$, then   \eqref{Steiner_formula} yields  \eqref{Nagy_O}. 
So we choose a ball $\Om^\cdast$ in the following way:
\begin{equation}\label{Ball_Quermass}
 \begin{aligned}
      \Om^\cdast\;\text{is the open ball centered at the origin such that}\;W_{n-1}(\Om^\cdast)=W_{n-1}(\Om).
\end{aligned}
\end{equation}
Indeed, we establish that $W_i(\Om)\le W_i(\Om^\cdast)$ for $i=1,2\ldots n-1$ (see Proposition \ref{Isoperimetric_quer}).

 \begin{remark}
 Notice that, for $n=2$, $W_{n-1}(\Om)=\frac{P(\Om)}{2}$ and hence $\Om^\cdast$ coincides with $\Om^\#.$ However, for $n>2$, $\Om^\#$ is a smaller ball than $\Om^\cdast$ (see Proposition \ref{Isoperimetric_quer}). 
 \end{remark}

 Let $\Om$ be as defined in \eqref{Domain} with $\Om_D=\Omi.$ For $n>2$ and $\Om_D$ convex, we consider a concentric annulus as below:
\begin{equation}\label{Dom23}
 \begin{aligned}
      &\widetilde{A}_I(\Om)=B_{R}\setminus \overline{B_{r}}\;\text{such that}\; |\Om|=|\widetilde{A}_I(\Om)|\;\text{and}\; W_{n-1}(\Om_D)=W_{n-1}(B_{r})=\om_n r.
\end{aligned}
\end{equation}
 Now we state the reverse Faber-Krahn type inequality for the first eigenvalue $\ta_{1,q}(\Om)$ of \eqref{Problem} when $\Om_D=\Omi$.
\begin{theorem}
\label{Theorem_In}
Let $\Om$ and $\widetilde{A}_I(\Om)$ be as given in \eqref{Domain} and \eqref{Dom23}, respectively, with $\Om_D=\normalfont{\Omi}$. For $q\in[1,p]$, let $\ta_{1,q}(\Om)$ be the first eigenvalue of \eqref{Problem} on $\Om$. If $\Om_D$ is convex, then
\begin{equation*}
    \ta_{1,q}(\Om)\leq \ta_{1,q}\left(\widetilde{A}_I(\Om)\right).
\end{equation*}
Furthermore, for $n\geq 3,$ the equality holds if and only if $\Om=\widetilde{A}_I(\Om)$ (up to a
translation).
\end{theorem}
\begin{remark}
For $n=2,$ since $A_I(\Om)=\widetilde{A}_I(\Om),$ the results obtained in \cite[Section 1.5]{Hersch} for convex $\Om_D$ will follow from the above theorem. If $\Om_D$ is a ball, then $A_I(\Om)=\widetilde{A}_I(\Om).$ Thus Theorem 1.2 of \cite{Anoop2020} is a special case of the above theorem.
\end{remark}

In \cite{Dellapietra}, Della Pietra and Piscitelli 
established  the inequality part in Theorem \ref{Theorem_In} (see \cite[Theorem 1.1, Remark 3.1]{Dellapietra}) using a web function and the co-area formula. Here we use the higher dimensional analogue (see Corollary \ref{Nagy_out_extension}) of Sz. Nagy's inequality for outer parallel sets of a convex set. In addition,
we prove the uniqueness of the maximizer in the admissible class of Lipschitz domains.

\begin{remark}($\Om_D$ is non-convex)
 To the best of our knowledge, the analogous results of the Theorem \ref{Theorem_Out} and Theorem \ref{Theorem_In} are not available in the literature for a non-convex domain  $\Om_D$ in $\R^n \;(n\geq 3).$ Here we provide various upper bounds for  $\ta_{1,q}(\Om)$ that work  for certain non-convex cases as well.
 \begin{enumerate}[(i)]
     \item \underline{\textit{Outer Dirichlet} ($\Om_D=\Om_{\text{out}}$)}:  Let $\Om=\Om_D\setminus \overline{\Omi}$ be as defined in \eqref{Domain}. Assume that $\Om$ is such that there exists an open ball $B_{\text{out}}$ such that $\overline{\Omi}\subsetneq B_{\text{out}}\subset \Om_D$. Let $\widetilde{\Om}=B_{\text{out}}\setminus\overline{\Omi}.$ Then Theorem \ref{Theorem_Out} yields that
     \begin{equation}\label{Non_cvx1}
    \ta_{1,q}(\widetilde{\Om})\leq \ta_{1,q}\left(A_O(\widetilde{\Om})\right),
\end{equation}
     where $A_O(\widetilde{\Om})=B_R\setminus \overline{B_r}$ satisfying $|\widetilde{\Om}|=|A_O(\widetilde{\Om})|$ and $P(B_{\text{out}})=P(B_R).$ Since $\widetilde{\Om}\subset \Om$, using domain monotonicity, we get $\ta_{1,q}(\Om)\leq \ta_{1,q}(\widetilde{\Om}).$ Hence using \eqref{Non_cvx1}, we obtain 
     $$\ta_{1,q}(\Om)\leq \ta_{1,q}\left(A_O(\widetilde{\Om})\right).$$
     If $\Om_D$ is a ball, then we can take $B_{\text{out}}=\Om_D$, then the above result coincides with \cite[Theorem 1.1]{Anoop2020}.
     
     \item \underline{\textit{Inner Dirichlet} ($\Om_D=\Om_{\text{in}}$)}: Suppose $\Om=\Omo\setminus \overline{\Om_D}$ is as defined in \eqref{Domain}.  Let $\text{Conv}(\Om_D)$ be the convex hull of $\Om_D$. Assume that $\overline{\text{Conv}(\Om_D)}\subsetneq \Om_{\text{out}}.$ Let $\widetilde{\Om}=\Omo\setminus\overline{\text{Conv}(\Om_D)}.$ Now by Theorem \ref{Theorem_In}, we have 
     \begin{equation}\label{Non_cvx2}
        \ta_{1,q}(\widetilde{\Om})\leq \ta_{1,q}\left(\widetilde{A}_I(\widetilde{\Om})\right),
     \end{equation}
     where $\widetilde{A}_I(\widetilde{\Om})=B_R\setminus \overline{B_r}$ satisfying $|\widetilde{\Om}|=|\widetilde{A}_I(\widetilde{\Om})|$ and $W_{n-1}(\text{Conv}(\Om_D))=W_{n-1}(B_r).$ Since $\Om_D\subset \text{Conv}(\Om_D),$ we have $\widetilde{\Om}\subset \Om$ and hence using domain monotonicity, we get $\ta_{1,q}(\Om)\leq \ta_{1,q}(\widetilde{\Om}).$ Therefore from \eqref{Non_cvx2}, we conclude
     $$\ta_{1,q}(\Om)\leq \ta_{1,q}\left(\widetilde{A}_I(\widetilde{\Om})\right).$$
 \end{enumerate}
\end{remark}

 The rest of the paper is organized as follows. In Section \ref{preli}, we state the existence results for the first eigenvalue of \eqref{Problem}. We prove the Sz. Nagy's type inequalities (Proposition \ref{Nagy_in_extension}, Theorem \ref{nag}) in Section \ref{ReverseNagy}. In Section \ref{targetOut}, we prove Theorem \ref{Theorem_Out}, and Section \ref{main} contains the proof of Theorem \ref{Theorem_In}. Finally, in Section \ref{remark}, we point out some consequences of the main results of this article and state a few open problems.

%% file: prelim.tex
\section{ Preliminaries}
\label{preli}
In this section, we prove the existence of the first eigenvalue $\ta_{1,q}$ of \eqref{Problem} and state a few properties of the first eigenfunctions.

\begin{proposition}\label{Existence}
For $q\in [1,p],$ let $\ta_{1,q}(\Om)$ be as defined in \eqref{variational_lambda}. Then there exists $u_q\in W^{1,p}_{\Ga_D}(\Om)$ such that
\begin{enumerate}[(i)]
    \item $\mathcal{R}_q(u_q)=\ta_{1,q}(\Om)$ and $u_q>0$ a.e. in $\Om.$
    \item $u_q\in C^{1,\ga}(\Om)$, where $0<\ga<1$.
    
    \item if $\mathcal{R}_q(v_q)=\ta_{1,q}(\Om),$ for some $v_q\in W^{1,p}_{\Ga_D}(\Om),$ then $u=cu_q$ for some $c\in \R.$

\end{enumerate}
\end{proposition}
\begin{proof}
 $(i)$ For $q\in [1,p],$ since the inclusion $W^{1,p}_{\Ga_D}(\Om)\hookrightarrow L^q(\Om)$ is compact, using the standard variational arguments (cf. \cite{Azorero_Garcia} or \cite[Proposition A.2]{Anoop2020}) we easily obtain that $\ta_{1,q}(\Om)>0$ and $\exists\; u_q\in W^{1,p}_{\Ga_D}(\Om)$ such that $\mathcal{R}_q(u_q)=\ta_{1,q}(\Om).$ Using the strong maximum principle \cite[Proposition 3.2]{Prashanth}, we also deduce that $u_q>0$ a.e. in $\Om$.
 
    $(ii)$ For a proof, see \cite[Theorem 1]{Lieber}.
    
    $(iii)$ This result is known for the complete Dirichlet problem; cf. \cite[Proposition 1.1]{Nazarov2001}. The same set of arguments work for the mixed boundary cases as well. 
    
  
\end{proof}

Next, we state a monotonicity result for the first eigenfunction of \eqref{Problem} on concentric annular regions. 
\begin{proposition}\label{Radiality}
Let $\Om=B_{\be}\setminus \overline{B_{\al}},$ for some $0<\al<\be$, where $B_{\al}$ and $B_{\be}$ are open balls centered at the origin of radius $\al$ and $\be$, respectively. Suppose that $q\in [1,p]$. Let $u_q>0$ be an eigenfunction of \eqref{Problem} associated to $\ta_{1,q}(\Om)$. Then $u_q$ is a radial function. Furthermore, the following holds:
\begin{enumerate}[(i)]
    \item (Outer Dirichlet) $u_q$ is strictly radially decreasing if $\Ga_D=\partial B_{\be}$,
    \item (inner Dirichlet) $u_q$ is strictly radially increasing if $\Ga_D=\partial B_{\al}$.
    \end{enumerate}
\end{proposition}
\begin{proof}
By Proposition \ref{Existence}, $\ta_{1,q}$ is simple, and hence radiality of $u_q$ follows immediately.

$(i)$ We adapt the ideas of proof of \cite[Proposition A.4]{Anoop2020} here.

Let $\Ga_D=\partial B_{\be}$. As $u_q$ is radial, $u_q(x)=f_q(|x|),$ for some function $f_q:\R\longrightarrow \R$. Therefore $f_q$ satisfies the following ordinary differential equation associated to \eqref{Problem}: 
\begin{equation*}
    \begin{aligned}
    -\left(|f_q'(r)|^{p-2} f_q'(r)r^{n-1}\right)' &=\ta_{1,q}(\Om)r^{n-1} f_q(r)^{q-1}\;\text{in}\; (\al,\be),\\
    f_q'(\al)&=0\;;\; f_q(\be)=0.
    \end{aligned}
\end{equation*}
Since $f_q(r)>0$ in $(\al,\be)$, we get $\left(|f_q'(r)|^{p-2} f_q'(r)r^{n-1}\right)'<0$ in $(\al,\be).$ Therefore, \\$|f_q'(r)|^{p-2} f_q'(r)r^{n-1}$ is strictly decreasing   and hence by the boundary condition at $\al$, we get $f_q'<0$ as required.

$(ii)$ Proof follows along the same line as in $(i).$
\end{proof}
\begin{remark}\label{Non_rad}
For $q>p,$  a first eigenfunction $u_q$ of \eqref{Problem} need not be radial on an arbitrary annular region $B_{\be}\setminus \overline{B_{\al}}$; see \cite[Proposition 1.2]{Nazarov2000} for such symmetry-breaking phenomena. Thus the assumption $q\in [1,p]$ in Proposition \ref{Radiality} can not be dropped.
\end{remark}

%% file: nagy.tex
\section{ Nagy's type inequality in higher dimensions}\label{ReverseNagy}
In this section, we prove the Sz. Nagy's type inequality for both outer parallel sets $\partial \Om_{\de}$ and inner parallel sets $\partial \Om_{-\de}$ for a convex domain $\Om$ in $\R^n$.  
The following  Alexandrov-Fenchel inequalities for quermassintegrals (cf. \cite[Section 6.4, (6.4.7)]{Schneider}) of a convex domain in $\R^n$ plays a vital role in our proofs.

\begin{proposition}[Alexandrov-Fenchel inequality]\label{Alexandrov_proposition}
Let $\Om\subset \R^n$ be a bounded convex domain. Then, 
 \begin{equation}
 	\label{Alexandrov_Fenchel}
 	\left(\frac{W_j(\Om)}{\omega_n}\right)^{\frac{1}{n-j}}\geq \left(\frac{W_i(\Om)}{\omega_n}\right)^{\frac{1}{n-i}},\;\;\;\text{for}\;\;0\leq i<j<n,
 \end{equation}
  and the equality holds for some $i$ and $j$, if and only if $\Om$ is a ball.
\end{proposition}   
Observe that, for $i=0,$ and $j=1,$ \eqref{Alexandrov_Fenchel} gives the following classical isoperimetric inequality \cite[Chapter 2]{Kesavan2006}:
  $$P(\Om)\ge n \omega_n^\frac{1}{n} |\Omega|^{\frac{n-1}{n}}.$$

Let $r_{\Om}$ be the inradius of $\Om$, i.e., $r_\Om$ is the supremum of the radii of all spheres that are
contained in $\Om$. 
We need the following lemma (cf. \cite[Lemma 3.1]{Brandolini}) to prove Proposition \ref{Nagy_in_extension}.
\begin{lemma}\label{Derivative_lemma}
Let $\Om\subset\mathbb{R}^n$ be a bounded, convex domain. Then
$$-\frac{d}{dt} P(\Om_{-t})\geq n(n-1)W_2(\Om_{-t}),\;\;\text{for a.e.}\; t\in(0,r_{\Om}).$$
If $\Om$ is a ball, then equality holds for every $t\in (0,r_{\Om})$.
\end{lemma}

Now we give a proof of Sz. Nagy's type inequality for inner parallels of a convex set in any space dimensions.
\begin{proof}[Proof of Proposition \ref{Nagy_in_extension}]
$(i)$ By Lemma \ref{Derivative_lemma} and using \eqref{Alexandrov_Fenchel} with $j=2$ and $i=1$, we get
\begin{align}\label{Alexandrov_1}
    -\frac{d}{dt} P(\Om_{-t}) &\geq n(n-1)W_2(\Om_{-t})\geq n^\frac{1}{n-1}(n-1)\om_n^\frac{1}{n-1}P(\Om_{-t})^\frac{n-2}{n-1}\;\text{for a.e.}\;t\in (0,r_\Om).
    \end{align}
Thus
    \begin{align*}\label{Om1}
     \frac{-\frac{d}{dt} P(\Om_{-t})}{P(\Om_{-t})^\frac{n-2}{n-1}} &\geq n^\frac{1}{n-1}(n-1)\om_n^\frac{1}{n-1},\;\text{for}\;t\in (0,r_\Om)\;\text{for a.e.}\;t\in (0,r_\Om).
\end{align*}
For the ball $\Om^\#$, the same computations as above yields
    \begin{equation*}\label{ann1}
     \frac{-\frac{d}{dt} P(\Om^\#_{-t})}{P(\Om^\#_{-t})^\frac{n-2}{n-1}} =n^\frac{1}{n-1}(n-1)\om_n^{\frac{1}{n-1}},\;\text{for every}\;t\in (0,r_\Om).
    \end{equation*}
Therefore, 
\begin{equation*}
    \frac{-\frac{d}{dt} P(\Om_{-t})}{P(\Om_{-t})^\frac{n-2}{n-1}}\geq \frac{-\frac{d}{dt} P(\Om^\#_{-t})}{P(\Om^\#_{-t})^\frac{n-2}{n-1}},\;\text{for a.e.}\;t\in (0,r_\Om).
\end{equation*}
Now for $\de\in (0,r_\Om)$, we integrate the above inequality from 0 to $\de$ to get
\begin{align}
     P(\Om)^{\frac{1}{n-1}}-P(\Om_{-\de})^\frac{1}{n-1}&\geq P(\Om^\#)^\frac{1}{n-1}-P(\Om^\#_{-\de})^\frac{1}{n-1},\;\text{for}\;\de\in (0,r_\Om).\nonumber
     \end{align}
    Since $P(\Om)=P(\Om^\#)$ and $\de$ is arbitrary, we obtain
     \begin{equation}\label{ngout}
     P(\Om_{-\de})\leq P(\Om^\#_{-\de}),\;\;\text{for all}\;\de\in (0,r_\Om).
\end{equation}

 $(ii)$ If $\Om$ is a ball, then $\Om=\Om^\#$ (upto translation) and hence equality holds in (\ref{ngout}). Conversely, if $P(\Om_{-\de})= P(\Om^\#_{-\de}),$ for some $\de\in (0,r_\Om)$, then by retracing the calculations that yield \eqref{ngout}, we conclude that equality must happen in \eqref{Alexandrov_1} for a.e. $t\in (0,\de)$. In particular, for some $t\in (0,\de)$, we have
\begin{align*}
    n(n-1)W_2(\Om_{-t})= n^\frac{1}{n-1}(n-1)\om_n^\frac{1}{n-1}P(\Om_{-t})^\frac{n-2}{n-1}.
    \end{align*}
    This yields
    \begin{equation*}
    \left(\frac{W_2(\Om_{-t})}{\omega_n}\right)^{\frac{1}{n-2}}= \left(\frac{W_1(\Om_{-t})}{\omega_n}\right)^{\frac{1}{n-1}}.
\end{equation*}
Therefore, by Proposition \ref{Alexandrov_proposition} (for $j=2$ and $i=1$), $\Om_{-t}$ must be a ball. Now $\Om=\left(\Om_{-t}\right)_t,$ and hence $\Om$ must be an open ball. This completes the proof.
\end{proof}

 Next, we give a proof of Theorem \ref{nag}.
 \begin{proof}[Proof of Theorem \ref{nag}]
  Let $R$ be the radius of $\Om^{\#}$. Since $P(\Om)=P(\Om^\#)$, we have
   $R= \left(\frac{P(\Om)}{n \omega_n}\right)^{\frac{1}{n-1}}$ and $W_1(\Om)=W_1(\Om^\#).$ Observe that
   \begin{equation*}\label{Ball}
       W_j(\Om^\#)=\omega_n^{\frac{1}{(n-j+1)}} \left( W_{j-1}(\Om^\#) \right)^{\frac{n-j}{n-j+1}}
   \end{equation*}
 Now,  by \eqref{Alexandrov_Fenchel} and the above identity, we get
 \begin{align}\label{compare1}
 	 W_2(\Om) \geq  \omega_n^{\frac{1}{n-1}} \left(W_1(\Om)\right)^{\frac{n-2}{n-1}}
 	 = W_2(\Om^{\#}).
 	 \end{align}
  Proceeding in this way, we get
  \begin{align}
  W_j(\Om) \geq \omega_n^{\frac{1}{(n-j+1)}} W_{j-1}(\Om)^{\frac{n-j}{n-j+1}}&\geq \omega_n^{\frac{1}{(n-j+1)}} \left( W_{j-1}(\Om^\#)\right)^{\frac{n-j}{n-j+1}}= W_j(\Om^{\#}),\;\;\text{for}\;2\leq j <n.\nonumber
  \end{align}
  Thus
  \begin{equation}\label{compareF}
   W_j(\Om) \geq W_j(\Om^{\#}), \;\text{for}\;1\leq j <n.
  \end{equation}
  Now by Steiner formula (\ref{Steiner_formula}), we get
  \begin{equation*}
  \label{nazyopp}
  \begin{split}
   P(\Om_{\de})
        = n \sum_{i=0}^{n-1} \binom {n-1}i W_{i+1}(\Om)\delta^i
  			& \geq n \sum_{i=0}^{n-1} \binom {n-1}i W_{i+1}(\Om^{\#})\delta^i\\
  			&=P(\Om^\#_{\de}).
  \end{split}
  \end{equation*}
  Conversely, if $\Om$ is not a ball, then by Proposition \ref{Alexandrov_proposition}, strict inequality occurs in \eqref{compareF} for all $2\leq j<n$. Thus by \eqref{Steiner_formula}, $P(\Om_{\de}) > P(\Om^\#_{\de}).$ 
  \end{proof}
First, we prove the following proposition. 

\begin{proposition}\label{Isoperimetric_quer}
Let $\Om\subset\mathbb{R}^n$ be a bounded, convex domain and $\Om^\cdast$ be as defined in \eqref{Ball_Quermass}. Then 
$$W_i(\Om)\le W_i(\Om^\cdast),\;\;\text{for all}\;\;0\leq i<n-1.$$
Moreover, equality holds if and only if $\Om$ is a ball.
\end{proposition}
\begin{proof}
Since $W_{n-1}(\Om)=W_{n-1}(\Om^{\cdast}),$ by taking $j=n-1$ in \eqref{Alexandrov_Fenchel}, for $i=0,1,\dots,n-2,$ we get
\begin{align}
	\left(\frac{W_i(\Om)}{\omega_n}\right)^{\frac{1}{n-i}} &\leq \frac{W_{n-1}(\Om)}{\omega_n}\nonumber
	= \frac{W_{n-1}(\Om^\cdast)}{\omega_n}\nonumber
	=\left(\frac{W_i(\Om^{\cdast})}{\omega_n}\right)^{\frac{1}{n-i}}.\nonumber\end{align}
	Therefore
	\begin{equation*}
	 \quad W_i(\Om) \leq W_i(\Om^{\cdast}), \quad\quad \text{for}\quad 0\leq i <n-1,
\end{equation*}
and equality occurs if and only if $\Om$ is a ball. This completes the proof.
\end{proof}

Now, as a consequence of the above Proposition, we get the following result that gives a Sz. Nagy's type inequality for the outer parallel sets of a convex domain in higher dimensions.
\begin{corollary}\label{Nagy_out_extension}
 Let $\Om\subset\mathbb{R}^n$ be a bounded, convex domain and $\Om^{\cdast}$ be as defined in \eqref{Ball_Quermass}. Then
\begin{equation*}\label{Nagy_higher}
    P(\Om_{\de}) \leq P(\Om^\cdast_{\de}),\;\;\text{for every}\;\de>0.
\end{equation*}
 Furthermore, if $n\geq 3$, then the equality holds in the above inequality if and only if $\Om$ is a ball.
\end{corollary}
\begin{proof}
Using Steiner formula \eqref{Steiner_formula} and Proposition \ref{Isoperimetric_quer}, we have
	\begin{align}
		P(\Om_{\de})= n \sum_{i=0}^{n-1} \binom {n-1}i W_{i+1}(\Om)\delta^i\nonumber
		&\leq n \sum_{i=0}^{n-1} \binom {n-1}i W_{i+1}(\Om^{\cdast})\delta^i= P(\Om^\cdast_{\de}).
	\end{align}
	If $P(\Om_{\de}) = P(\Om^\cdast_{\de})$, then $W_i(\Om)=W_i(\Om^\cdast)$, for all $1\leq i\leq n-1$. Thus by Proposition \ref{Isoperimetric_quer}, $\Om$ must be a ball.
\end{proof}
\begin{remark}
Notice that, if $n=2$ and $\Om$ is convex, then by Steiner formula (see \cite{Steiner} or \cite[Theorem 10.1]{Gray}), we have $P(\Om_{\de}) = P(\Om^\cdast_{\de})$ for every $\de >0$.
\end{remark}

%% file: outerD.tex
\section{Outer Dirichlet problem}
\label{targetOut}
In this section, we prove Theorem \ref{Theorem_Out}. Let $\Om$ be as stated in \eqref{Domain}, $\Om_D=\Omo$ and $A_O(\Om)=B_{R}\setminus \overline{B_{r}}$ be the annulus  having the same volume as $\Om$ such that $P(\Om_D)=P(B_R)$. 

Define $$t^*=\sup\{t>0: P( \Om_{D_{-t}} \cap \Om)>0\},$$ $$s(\de)=P( \Om_{D_{-\de}} \cap \Om),\;\text{for}\;\de\in (0,t^*]\;\text{and}\;S(\de)= P(B_{R_{-\de}}),\;\text{for}\;\de\in (0,R-r].$$ 
\begin{remark}\label{Nagy_rem1}
Using Proposition \ref{Nagy_in_extension}, we obtain
\begin{equation}\label{s_delta1}
    s(\delta)\leq P(\Om_{D_{-\de}})\leq P(B_{R_{-\de}}) =S(\delta),\;\text{ for all}\; \delta\in[0,R-r].
\end{equation}
In \cite[Lemma 2.3]{Anoop2020}, the above inequality follows easily from the fact that $\Om_D$ is a  ball. 
\end{remark}
Let $v(\de)$ be the volume of the portion of $\Om$ lying between $\partial \Om_{D_{-\de}}$ and $\partial\Om_D$, and  $V(\de)$ be the volume of the portion of $A_O(\Om)$ lying between $\partial B_{R_{-\de}}$ and $\partial B_{R}$. Then
\begin{equation}\label{aAdelta}
    v(\de)=\displaystyle\int_0^\de s(t)\dt,\;\text{where}\;\de\in [0,t^*],\quad V(\de)=\displaystyle\int_0^\de S(t)\dt,\;\text{where}\;\de\in [0,R-r].
\end{equation}
Observe that, both $v$ and $V$ are strictly increasing, and
\begin{align*}
    v(\delta)\le V(\delta), \forall\, \delta \in [0,R-r],\\
    v(t^*)=|\Om|=|A_O(\Om)|=V(R-r).
\end{align*}
For $\al \in [0,|\Om|]$, define $h$ and $H$ as follows:
\begin{equation*}
    h(\al)=s(v^{-1}(\al)),\quad H(\al)=S(V^{-1}(\al)).
\end{equation*}
Next, we prove an auxiliary lemma that was proved in \cite[Lemma 2.5]{Anoop2020} under the assumption that $\Om_D$ is a ball.
\begin{lemma}\label{aux2} Let $\Om$ and $A_O(\Om)$ be as mentioned in \eqref{Domain} and \ref{AO}, respectively, $\Om_D$ convex. Then the following holds:
\begin{enumerate}[(i)]
    \item $t^*\geq R-r$. Furthermore, for $n\geq3,$ equality occurs if and only if $\Om=A_O(\Om)$ (up to a translation),
    
    \item  $h(\al)\leq H(\al)$, for all $\al \in [0,|\Om|]$. Moreover, for $n\geq3,$ $h(|\Om|)=H(|\Om|)$ if and only if $\Om=A_O(\Om)$ (up to a translation). Furthermore, if $\Om$ is not a concentric annulus, then there exists $\al_0\in(0,|\Om|)$ such that $h(\al)< H(\al)$, for all $\al \in [\al_0,|\Om|]$.
\end{enumerate}
\end{lemma}
\begin{proof}

$(i)$ If possible, let $t^*<R-r$. Since $s(\de)\leq S(\de)$ (by \eqref{s_delta1}), we obtain
\begin{align*}
|A_O(\Om)|=\int_{0}^{R-r}S(\delta)\dd&\geq	\int_{0}^{t^*}s(\delta)\dd+\int_{t^*}^{R-r}S(\delta)\dd=|\Omega|+\int_{t^*}^{R-r}S(\delta)\dd>|\Omega|,
	\end{align*}
a contradiction to the fact that $|\Om|=|A_O(\Om)|$. Thus $t^*\geq R-r.$ If $t^*= R-r$, then $$ \displaystyle\int_{0}^{R-r}(S(\delta)-s(\delta))\dd=|A_O(\Om)|-|\Om|=0.$$
Since $s$ and $S$ are continuous, and $s(\de)\leq S(\de)$, we deduce that $s(\delta)=S(\delta),$ for all $\delta\in[0,R-r]$. Thus by \eqref{s_delta1}, we get $P(\Om_{D_{-\de}})= P(B_{R_{-\de}})$. Therefore, using Proposition \ref{Nagy_in_extension}, we conclude $\Om_D=B_{R}$ (up to a translation). Since $t^*= R-r$, by the definition of $t^*,$ we must have $P(B_{R-t}\cap \Om)=0,$ for all $t>R-r,$ i.e., $P(B_s\cap (B_R\setminus\Omi))=0,$ for all $s<r$. Thus
$$|B_r\setminus \Omi|=0.$$
Notice that, 
\begin{align*}
    |B_R\setminus B_r| &=|B_R\setminus (\Omi\cup B_r)|+ |\Omi\setminus B_r|,\\
    \text{and}\;\;|B_R\setminus \Omi| &=|B_R\setminus (\Omi\cup B_r)|+ |B_r\setminus \Omi|.
\end{align*}
Since $|B_R\setminus B_r|=|B_R\setminus \Omi|$ and $|B_r\setminus \Omi|=0,$ we get $|\Omi\setminus B_r|=0.$  Since $\Omi$ is Lipschitz (as $\Om$ is), we can prove that either $\Omi=B_r$ or $|(\Omi\setminus B_r)\cup(B_r\setminus \Omi)|>0.$ As $|(\Omi\setminus B_r)\cup(B_r\setminus \Omi)|=0,$ we conclude that $\Om=A_O(\Om)$ (up to a translation).

$(ii)$ Since  $v^{-1}(\al)\geq V^{-1}(\al)$ and $S$ is a decreasing function, we obtain $$h(\al)=s(v^{-1}(\al))\leq S(v^{-1}(\al))\leq S(V^{-1}(\al))=H(\al).$$
If $h(|\Om|)=H(|\Om|),$ then $s(t^*)=S(R-r).$  Thus we must have $t^*=R-r$. Otherwise, if $t^*>R-r$, then
$$h(|\Om|)=s(t^*)\leq S(t^*)<S(R-r)=H(|\Om|),$$
a contradiction. Therefore the equality case follows from $(i).$ Now if $\Om$ is not a concentric annulus, then $t^*>R-r$. Set $\de_0 =R-r.$ Clearly, $v(\de_0)<V(\de_0)$. If not, then $$|A_O(\Om)|=V(R-r)=v(R-r)<v(t^*)=|\Om|,$$  a contradiction. Now from the definition of $v$ and $V$, we have $$v(\de)=v(\de_0)+\displaystyle\int_{\de_0}^{\de}s(t)\d t<V(\de)\;\;\text{ for all}\; \de \in [\de_0,t^*].$$
Therefore taking $\al_0=V(\de_0)$ and using the above inequality, we obtain $v^{-1}(\al)>V^{-1}(\al)$ for all $\al\in [\al_0, |\Om|].$ Since $S$ is strictly decreasing, we get
$$h(\al)=s(v^{-1}(\al))\leq S(v^{-1}(\al))< S(V^{-1}(\al))=H(\al)\;\;\text{for all}\; \al\in [\al_0, |\Om|].$$
\end{proof}
Now we give a proof of Theorem \ref{Theorem_Out}. Our proof for the inequality part is similar to the proof of Theorem 1.1 in \cite{Anoop2020} except at a few lines. For the sake of completeness, we supply a proof.
\begin{proof}[Proof of Theorem \ref{Theorem_Out}]
For $q\in [1,p]$, let $w_q$ be a positive eigenfunction associated to\\ $\ta_{1,q}(A_O(\Om))$. Now by Proposition \ref{Radiality}, $w_q$ is a radial function in $A_O(\Om)$. Also $w_q\in C^{1,\ga}(A_O(\Om)),$ for some $0<\ga<1$ (using Proposition \ref{Existence}-$(ii)$). Therefore, we can choose $\phi_q\in C^1(\mathbb{R})$ with $\phi_q(0)=0$ such that $w_q$ can be represented as $w_q(x)=\phi_q(R-|x|)$, for all $x\in A_O(\Om)$. Set $\psi_q=\phi_q \circ V^{-1},$ where $V$ is as defined in \eqref{aAdelta}. Notice that $\psi_q(0)=0$. Let $\rho_1$ and $\rho_2$ be the distance functions from $\pa B_R$ and $\Ga_D,$ respectively, defined as follows:
$$\rho_1(x)=d(x,\pa B_R),\;  \rho_2(x)=d(x,\Ga_D)\;\;\text{ for all}\; x\in \R^n.$$
Therefore $w_q(x)=\psi_q(V(\rho_1(x)))$ for all $x\in A_O(\Om)$. Now we define a test function $u_q$ on $\Om$ in the following way:
\begin{equation*}
    u_q(x)=\psi_q(v(\rho_2(x))),\; \text{for all}\;x\in \Om,
\end{equation*}
where $v$ is given by \eqref{aAdelta}. Then $u_q\in W^{1,p}(\Om)$ and $u_q$ vanishes on $\Ga_D$. Using the fact that $|\nabla \rho_2 (x)|=1,$ (cf. \cite[Theorem 3.14]{Evans}) and by the Coarea formula (cf. \cite[Appendix C, Theorem 5]{EvansPde}), we get
\begin{align}\label{domg}
\displaystyle\int_{\Om}|\nabla u_q(x)|^p \dx&=\int_{0}^{t^*}\int_{(\pa\Om_{D_{-\de}})\cap \Om}|\nabla u_q(x)|^p\dsi \dd \no\\ &=\int_{0}^{t^*}\int_{(\pa\Om_{D_{-\de}})\cap \Om}\big\{|\psi_q'(v(\rho_2(x)))|s(\rho_2(x))\big\}^p\dsi \dd \no\\
&=\int_{0}^{t^*}|\psi_q'(v(\de))|^p s(\de)^{p+1}\dd. \no
\end{align}
Now make a change of variable $\al=v(\de)$ in the above expression to get
\begin{equation}\label{domg}
\displaystyle\int_{\Om}|\nabla u_q(x)|^p \dx =\int_{0}^{|\Om|}|\psi_q'(\al)|^p h(\al)^p \da.
\end{equation}
 In a similar manner, we obtain
\begin{equation}\label{annug}
    \displaystyle\int_{A_O(\Om)}|\nabla w_q(x)|^p \dx=\int_{0}^{|A_O(\Om)|}|\psi_q'(\al)|^p H(\al)^p \da.
\end{equation}
Therefore, Lemma \ref{aux2}-$(ii)$, along with \eqref{domg} and \eqref{annug}, yields 
\begin{equation}\label{gradq}
    \displaystyle\int_{\Om}|\nabla u_q(x)|^p \dx\leq\int_{A_O(\Om)}|\nabla w_q(x)|^p \dx.
\end{equation}
Now by performing similar computations, we get
\begin{equation}\label{Lp}
    \displaystyle\int_{\Om}|u_q(x)|^q \dx=\int_{0}^{|\Om|}|\psi_q(\al)|^q \da=\int_{0}^{|A_O(\Om)|}|\psi_q(\al)|^q \da=\int_{A_O(\Om)}|w_q(x)|^q \dx.
\end{equation}
Now using \eqref{gradq} and \eqref{Lp} in the variational characterization \eqref{variational_lambda} of $\ta_{1,q},$ we conclude that
$$\ta_{1,q}(\Om)\leq \ta_{1,q}({A_O(\Om)}).$$
If $\ta_{1,q}(\Om)=\ta_{1,q}({A_O(\Om)}),$ then we have the equality in \eqref{gradq} and hence from \eqref{domg} and \eqref{annug}, we conclude that $h(\al)=H(\al)$ for all $\al\in [0,|\Om|]$. Thus by Lemma \ref{aux2}-$(ii)$, $\Om=A_O(\Om)$ (up to a translation). This completes the proof.
\end{proof}

%% file: innerD.tex
\section{Inner Dirichlet problem}\label{main}
In this section, we prove Theorem \ref{Theorem_In}. Let $\Omega$ be as stated in \eqref{Domain} with $\Om_D=\Omi$ and $\widetilde{A}_I(\Om)=B_{R}\setminus \overline{B_{r}}$ is the concentric annulus with the same volume as $\Om$ and $W_{n-1}(\Om_D)=W_{n-1}(B_{r})$.

Let	us define 
$$\delta_*=\sup\{\delta>0: P(\Om_{D_\de}\cap \Omega) >0 \},\;s(\delta)=P(\Om_{D_\de}\cap \Omega),\;\text{for}\;\delta\in (0,\delta_*],$$
$$ \text{and}\;S(\delta)= P(B_{r_\de}),\;\text{for}\;\de\in (0,R-r].$$
\begin{remark}\label{Nagy_rem2}
By applying Corollary \ref{Nagy_out_extension}, we have
\begin{equation}\label{s_delta2}
    s(\de)\leq P(\Om_{D_\de})\leq P(B_{r_\de})=S(\de),\;\text{for all}\;\de \in [0,R-r].
\end{equation}
However, in \cite[Section 2.2]{Anoop2020}, authors obtained the above inequality easily by assuming that $\Om_D$ is a ball.
\end{remark}
For $p\in (1,\infty)$, consider the parametrizations $t$ and $T$, similar to Hersch \cite{Hersch} for $p=2$ and \cite{Anoop2020} for $p\neq 2$, as below
	\begin{equation}\label{parame}
		t(\delta)=\displaystyle\int_{0}^{\delta}\frac{1}{s(\rho)^{p'-1}}\drho ,\quad 
		T(\delta)=\int_{0}^{\delta}\frac{1}{S(\rho)^{p'-1}}\drho,
	\end{equation}
	where $p'=\frac{p}{p-1}$ is the holder conjugate of $p$. Let $t_*=t(\delta_*)$ and $T_{\#}=T(R-r)$. Define $$g(\alpha)=s(t^{-1}(\alpha)),\;\text{for}\;\alpha \in [0,t_*]\;\text{and}\; G(\alpha)=S(T^{-1}(\alpha)),\;\text{for}\;\alpha \in [0,T_{\#}].$$
	
	Next, we state a technical lemma without proof that will be used to prove the main results. This lemma was proved in \cite[Lemma 2.7]{Anoop2020} for the case when $\Om_D$ is a ball. The similar proof can be carried out for $\Om_D$ convex also with the help of \eqref{s_delta2}.
	\begin{lemma}\label{auxi} Let $\Omega $ and $\widetilde{A}_I(\Om)$ be as stated in \eqref{Domain} and \eqref{Dom23}, respectively, with $\Om_D$ convex. Then we have the following:
		\begin{enumerate}[(i)]
			\item  $R-r\leq \delta_*$ and $T_{\#}\leq t_*$. Furthermore, for $n\geq3,$ $R-r=\de_*$ or $T_{\#}= t_*$ if and only if $\Omega=\widetilde{A}_I(\Om)$ (up to a translation).
		\item  $g(\alpha)\leq G(\alpha)$, for all $\alpha \in [0,T_\#]$ and for $n\geq3,$ equality happens if and only if $\Omega=\widetilde{A}_I(\Om)$ (up to a translation). Furthermore, if $\Omega$ is not a concentric annulus, then $g(\alpha)<G(\alpha)$ on $(\alpha',T_\#]$, for some $\alpha'\in[0,T_\#]$.
			\item $\displaystyle\int_{0}^{t_*}g(\alpha)^{p'}\da=\int_{0}^{T_\#}G(\alpha)^{p'}\da=|\Omega|.$
		\end{enumerate}
	\end{lemma}
		
		
 Now we give a proof of Theorem \ref{Theorem_In}. We use the ideas of the proof of \cite[Theorem 1.2]{Anoop2020}. For the sake of completeness, we prove it here.
\begin{proof}[Proof of Theorem \ref{Theorem_In}]
 		Let $q\in [1,p]$ and $v_q$ be a positive eigenfunction corresonding to $\ta_{1,q}(\widetilde{A}_I(\Om))$. Then $v_q$ is radially symmetric and radially increasing in the concentric annular region $\widetilde{A}_I(\Om)$ (by Proposition \ref{Radiality}). Let $\rho_1$ and $\rho_2$ be the distance functions from $\pa B_r$ and $\Ga_D,$ respectively, defined as below
 		$$\rho_1(x)=d(x,\pa B_r),\;\; \rho_2(x)=d(x,\Ga_D)\;\;\text{ for all}\; x\in \R^n.$$
 		Now we can choose a $C^1$ function $\phi_q$ on $\mathbb{R}$ such that $\phi_q(0)=0$ and $v_q$ can be expressed as follows:
 		\begin{equation*}
 		v_q(x)=\phi_q(|x|-r)=\phi_q(\rho_1(x))=(\phi_q\circ T^{-1})(T(\rho_1(x)))=\psi_q(T(\rho_1(x))),
 		\end{equation*}
 		where $T$ is as stated in \eqref{parame} and $\psi_q:=\phi_q\circ T^{-1}$. Since $v_q$ is radially increasing in $\widetilde{A}_I(\Om)$, the maximum of $v_q$ will be on the outer boundary $\partial B_{R}$. Therefore, $\psi_q(T_\#)=\max\limits_{[0,T\#]}\psi_q $. 	Now we define a test function on $\Omega$ in the following way. Define
 		\begin{equation*}
 		u_q(x)= \begin{cases}
 		 \psi_q(t(\rho_2(x))),& \mbox{if } \quad t(\rho_2(x)) \in [0,T_\#], \\
 		\psi_q(T_\#), & \mbox{if} \quad t(\rho_2(x)) \in (T_\#,t_*], 
 		\end{cases}
 		\end{equation*}
 		where $t$ is as defined in \eqref{parame}. Observe that $u_q\in W^{1,p}(\Omega)$ and $u_q$ vanishes on $ \Ga_D$. Also $\nabla u_q(x)=0$, if $t(\rho_2(x)) \in (T_\#,t_*]$.
 		Therefore, using the fact that $|\nabla \rho_2(x)|=1, \;\forall\;x\in \Om$ and by the Coarea formula, we get
 		\begin{align*}
 			\int_{\Omega}|\nabla u_q(x)|^p\dx &=\int_{0}^{t^{-1}(T_\#)}\int_{(\pa\Om_{D_{\de}})\cap \Om}|\nabla u_q(t(\rho_2(x)))|^p\dsi \dd\\
 			&=\int_{0}^{t^{-1}(T_\#)}\frac{|\psi_q'(t(\delta))|^p}{s(\delta)^{p'-1}}\dd\\
 			&=\int_{0}^{T_\#}|\psi_q'(\alpha)|^p\da,\;\;(\text{by a change of variable}\; t(\delta)=\alpha).\;
 		\end{align*}
 		Hence 
 		\begin{equation*}
 		    \int_{\Omega}|\nabla u_q(x)|^p\dx=\int_{0}^{T_\#}|\psi_q'(\alpha)|^p\da.
 		\end{equation*}
 		Similary, we have $\displaystyle\int_{\widetilde{A}_I(\Om)}|\nabla v_q(x)|^p \dx=\int_{0}^{T_\#}|\psi_q'(\alpha)|^p\da$. Thus,
 		\begin{equation}\label{e2}
 		    \int_{\Omega}|\nabla u_q(x)|^p\dx=\displaystyle\int_{\widetilde{A}_I(\Om)}|\nabla v_q(x)|^p \dx.
 		\end{equation}
 		Also using the definition of $u_q$, we obtain
 		\begin{align*}
 			\displaystyle\int_{\Omega}|u_q(x)|^q \dx&=\int_{0}^{t_*}|\psi_q(\alpha)|^q g(\alpha)^{p'}\da\\
 			&=\int_{0}^{T_\#}|\psi_q(\alpha)|^q g(\alpha)^{p'}\da+ |\psi_q(T_\#)|^q\int_{T_\#}^{t_*}g(\alpha)^{p'}\da
 		\end{align*}
 		and 
 		\begin{align*}
 			\displaystyle\int_{\widetilde{A}_I(\Om)}|v_q(x)|^q \dx&=\int_{0}^{T_\#}|\psi_q(\alpha)|^q G(\alpha)^{p'}\da.
 		\end{align*}
 		Therefore, by Lemma \ref{auxi}-$(ii)$ and using the fact that $\psi(T_\#)\geq \psi(\alpha), \;\forall\;\al\in[0,T_\#]$, we have
 		\begin{align*}
 			\displaystyle\int_{\widetilde{A}_I(\Om)}|v_q(x)|^q \dx-\displaystyle\int_{\Omega}|u_q(x)|^q \dx &\leq |\psi_q(T_\#)|^q\Bigg\{\int_{0}^{T_\#}(G(\alpha)^{p'}-g(\alpha)^{p'})\da-\int_{T_\#}^{t_*} g(\alpha)^{p'}\da\Bigg\}\nonumber\\
 			&\leq 0.
 			\end{align*}
 			Thus
 			\begin{equation}\label{e3}\displaystyle\int_{\widetilde{A}_I(\Om)}|v_q(x)|^q \dx\leq \displaystyle\int_{\Omega}|u_q(x)|^q \dx.
 		\end{equation}
 		Hence the conclusion follows using \eqref{e2} and \eqref{e3} in variational characterization \eqref{variational_lambda} of $\ta_{1,q}$. The equality case follows using similar arguments as in the proof of Theorem \ref{Theorem_Out}. This concludes the proof.
 		\end{proof}

%% file: remarks.tex
\section{Some remarks and open problems}
\label{remark}

\noi\textbf{Torsional rigidity:}
The arguments used in this article can be adapted to study the torsion problem on multiply-connected domains with mixed boundary conditions. Let $\Om$, $A_O(\Om)$ and $\widetilde{A}_I(\Om)$ be as stated in \eqref{Domain}, \ref{AO}, and \eqref{Dom23}, respectively. For $p\in (1,\infty)$, consider the following problem:
\begin{equation}\tag{T}\label{Torsion_problem}
	\left. \begin{aligned}
	-\Delta_p u &=1 \qquad\text{in} \quad\Om,\\
	u &=0 \qquad\text{on}\quad \Ga_D,\\
	\frac{\partial u}{\partial \eta} &=0 \qquad\text{on} \quad \pa \Om\setminus\Ga_D.
	\end{aligned}\right\}
\end{equation}
Now the variational characterization of $p$-torsional rigidity $T(\Om)$ of \eqref{Torsion_problem} is given by
\begin{equation}
    \frac{1}{T(\Om)}=\displaystyle\inf\Bigg\{\frac{\int_{\Omega}|\nabla u|^p \dx}{\big(\int_{\Omega}|u| \dx\big)^p}: u\in W^{1,p}_{\Ga_D}(\Omega)\setminus\{0\} \Bigg\}.
\end{equation}
Then using a similar test function as in the proof of Theorem \ref{Theorem_Out} and Theorem \ref{Theorem_In} respectively, we can establish the following results.
\begin{theorem}\label{Torsion_Out}
Let $\Om$ and $A_O(\Om)$ be as in Theorem \ref{Theorem_Out}. Assume that $\Om_D=\normalfont{\Omo}$ and $\Om_D$ is convex. Then $T(A_O(\Om))\leq T(\Om)$ and for $n\geq3,$ equality holds if and only if $\Om=A_O(\Om)$ (up to a translation). 
\end{theorem}
\begin{theorem}\label{Torsion_In}
Let $\Om$ and $\widetilde{A}_I(\Om)$ be as in Theorem \ref{Theorem_In}. Assume that $\Om_D=\normalfont{\Omi}$ and $\Om_D$ is convex. Then $T(\widetilde{A}_I(\Om))\leq T(\Om)$ and for $n\geq3,$ equality holds if and only if $\Om=\widetilde{A}_I(\Om)$ (up to a translation). 
\end{theorem}

\noi\textbf{Open problems:} Now we state a few open problems related to Sz. Nagy's inequalities and the reverse Faber-Krahn inequalities discussed here. 
\begin{enumerate}[(i)]

\item \textit{Sz. Nagy's inequality for non-convex domains}: We have extended Sz.  Nagy's inequalities \eqref{Nagy_in} and \eqref{Nagy_out} to higher dimensions (Proposition \ref{Nagy_in_extension} and Corollary \ref{Nagy_out_extension}, respectively) for convex domains. The analogue of these results for the non-convex domains in higher dimensions are not known. 

    \item
 Let $p^*=\frac{np}{n-p}$ and $q\in (p,p^*)$. By the compactness of the Sobolev embedding  $W^{1,p}_{\Ga_D}(\Om)\hookrightarrow L^q(\Om),$ $\ta_{1,q}(\Om)$ is attained for some $u\in W^{1,p}_{\Ga_D}(\Om)$. However, as we pointed out in Remark \ref{Non_rad}, $u$ need not be radial on an arbitrary concentric annular region. Therefore, our method of proof is not applicable for $q\in (p,p^*)$ and the analogue of Theorem \ref{Theorem_Out} and \ref{Theorem_In} seems to be a challenging open problem in these cases.

    
    \item We give an analogue of Hersch's result \eqref{Hersch_rfk} in higher dimensions using the  Sz. Nagy’s type  inequality (Corollary \ref{Nagy_out_extension}) for
outer parallel sets with quermassintegral constraint. The extension of \eqref{Hersch_rfk} to higher dimensions with respect to the perimeter or any other quermassintegral constraint on the Dirichlet boundary is entirely open. 

\item Though Sz. Nagy's inequality fails (see Theorem \ref{nag}) with the perimeter constraints, the reverse Faber-Krahn inequality \eqref{Hersch_rfk} for the inner Dirichlet problem still holds for a certain convex domain in higher dimensions. We provide a numerical example using COMSOLE MULTIPHYSICS (Version 4.3).

Let \begin{align*}
    \Omo &=\{(x,y,z)\in \R^3: |x|<0.5, |y|<0.75, |z|<1\},\\
    \Omi &=\{(x,y,z)\in \R^3: |x|<0.4, |y|<0.65, |z|<0.9\}.
\end{align*}
Suppose $\Om=\Omo\setminus\overline{\Omi}$ and $\Om_D=\Omi.$ Let $A_I(\Om)=B_R\setminus\overline{B_r}$ be the concentric annular region such that $|A_I(\Om)|= |\Om|$ and $P(B_r)= P(\Om_D).$ Then the following reverse Faber-Krahn inequality holds (approximately)
$$\ta_{1,2}(\Om)\approx 0.23429< 0.87586\approx\ta_{1,2}(A_I(\Om)).$$
However, Sz. Nagy's inequality for the outer parallel sets of $\Om_D$ fails with the perimeter constraint (see Theorem \ref{nag}). Thus a different approach for proving the
reverse Faber-Krahn inequality that applies to this kind of domain can
be explored.

\end{enumerate}

%% file: R-F-K.bbl
\begin{thebibliography}{10}

\bibitem{Anoop2020}
T.~V. Anoop and K.~Ashok~Kumar.
\newblock On reverse {F}aber-{K}rahn inequalities.
\newblock {\em J. Math. Anal. Appl.}, 485(1):123766, 20, 2020.
\newblock \href {https://doi.org/10.1016/j.jmaa.2019.123766}
  {\path{doi:10.1016/j.jmaa.2019.123766}}.

\bibitem{anoop2018}
T.~V. Anoop, V.~Bobkov, and S.~Sasi.
\newblock On the strict monotonicity of the first eigenvalue of the
  {$p$}-{L}aplacian on annuli.
\newblock {\em Trans. Amer. Math. Soc.}, 370(10):7181--7199, 2018.
\newblock \href {https://doi.org/10.1090/tran/7241}
  {\path{doi:10.1090/tran/7241}}.

\bibitem{Bobkov2020Kolonitskii}
V.~Bobkov and S.~Kolonitskii.
\newblock On qualitative properties of solutions for elliptic problems with the
  {$p$}-{L}aplacian through domain perturbations.
\newblock {\em Comm. Partial Differential Equations}, 45(3):230--252, 2020.
\newblock \href {https://doi.org/10.1080/03605302.2019.1670674}
  {\path{doi:10.1080/03605302.2019.1670674}}.

\bibitem{Brandolini}
B.~Brandolini, C.~Nitsch, and C.~Trombetti.
\newblock An upper bound for nonlinear eigenvalues on convex domains by means
  of the isoperimetric deficit.
\newblock {\em Arch. Math. (Basel)}, 94(4):391--400, 2010.
\newblock \href {https://doi.org/10.1007/s00013-010-0102-8}
  {\path{doi:10.1007/s00013-010-0102-8}}.

\bibitem{Bucur2015Giacomini}
D.~Bucur and A.~Giacomini.
\newblock Faber-{K}rahn inequalities for the {R}obin-{L}aplacian: a free
  discontinuity approach.
\newblock {\em Arch. Ration. Mech. Anal.}, 218(2):757--824, 2015.
\newblock \href {https://doi.org/10.1007/s00205-015-0872-z}
  {\path{doi:10.1007/s00205-015-0872-z}}.

\bibitem{Burago}
Y.~D. Burago and V.~A. Zalgaller.
\newblock {\em Geometric inequalities}, volume 285 of {\em Grundlehren der
  Mathematischen Wissenschaften [Fundamental Principles of Mathematical
  Sciences]}.
\newblock Springer-Verlag, Berlin, 1988.
\newblock Translated from the Russian by A. B. Sosinski\u{\i}, Springer Series
  in Soviet Mathematics.
\newblock \href {https://doi.org/10.1007/978-3-662-07441-1}
  {\path{doi:10.1007/978-3-662-07441-1}}.

\bibitem{AnisaMrityunjoy}
A.~M.~H. Chorwadwala and M.~Ghosh.
\newblock Optimal shapes for the first {D}irichlet eigenvalue of the
  {$p$}-{L}aplacian and dihedral symmetry.
\newblock {\em J. Math. Anal. Appl.}, 508(2):Paper No. 125901, 18, 2022.
\newblock \href {https://doi.org/10.1016/j.jmaa.2021.125901}
  {\path{doi:10.1016/j.jmaa.2021.125901}}.

\bibitem{Daners2007}
D.~Daners and J.~Kennedy.
\newblock Uniqueness in the {F}aber-{K}rahn inequality for {R}obin problems.
\newblock {\em SIAM J. Math. Anal.}, 39(4):1191--1207, 2007/08.
\newblock \href {https://doi.org/10.1137/060675629}
  {\path{doi:10.1137/060675629}}.

\bibitem{Dellapietra}
F.~Della~Pietra and G.~Piscitelli.
\newblock An optimal bound for nonlinear eigenvalues and torsional rigidity on
  domains with holes.
\newblock {\em Milan J. Math.}, 88(2):373--384, 2020.
\newblock \href {https://doi.org/10.1007/s00032-020-00320-9}
  {\path{doi:10.1007/s00032-020-00320-9}}.

\bibitem{EvansPde}
L.~C. Evans.
\newblock {\em Partial differential equations}, volume~19 of {\em Graduate
  Studies in Mathematics}.
\newblock American Mathematical Society, Providence, RI, second edition, 2010.
\newblock \href {https://doi.org/10.1090/gsm/019} {\path{doi:10.1090/gsm/019}}.

\bibitem{Evans}
L.~C. Evans and R.~F. Gariepy.
\newblock {\em Measure theory and fine properties of functions}.
\newblock Textbooks in Mathematics. CRC Press, Boca Raton, FL, revised edition,
  2015.

\bibitem{Faber}
G.~Faber.
\newblock {\em Beweis, dass unter allen homogenen Membranen von gleicher
  Fl{\"a}che und gleicher Spannung die kreisf{\"o}rmige den tiefsten Grundton
  gibt}.
\newblock Verlagd. Bayer. Akad. d. Wiss., 1923.

\bibitem{Azorero_Garcia}
J.~P. Garc\'{\i}a~Azorero and I.~Peral~Alonso.
\newblock Existence and nonuniqueness for the {$p$}-{L}aplacian: nonlinear
  eigenvalues.
\newblock {\em Comm. Partial Differential Equations}, 12(12):1389--1430, 1987.
\newblock \href {https://doi.org/10.1080/03605308708820534}
  {\path{doi:10.1080/03605308708820534}}.

\bibitem{Gray}
A.~Gray.
\newblock {\em Tubes}, volume 221 of {\em Progress in Mathematics}.
\newblock Birkh\"{a}user Verlag, Basel, second edition, 2004.
\newblock With a preface by Vicente Miquel.
\newblock \href {https://doi.org/10.1007/978-3-0348-7966-8}
  {\path{doi:10.1007/978-3-0348-7966-8}}.

\bibitem{henrot2021}
A.~Henrot.
\newblock {\em Shape optimization and spectral theory}.
\newblock De Gruyter, Berlin, Boston, 13 Apr. 2021.
\newblock URL: \url{https://www.degruyter.com/view/title/530007}, \href
  {https://doi.org/https://doi.org/10.1515/9783110550887}
  {\path{doi:https://doi.org/10.1515/9783110550887}}.

\bibitem{Henrot2006}
A.~Henrot.
\newblock {\em Extremum problems for eigenvalues of elliptic operators}.
\newblock Frontiers in Mathematics. Birkh\"{a}user Verlag, Basel, 2006.

\bibitem{Hersch}
J.~Hersch.
\newblock The method of interior parallels applied to polygonal or multiply
  connected membranes.
\newblock {\em Pacific J. Math.}, 13:1229--1238, 1963.
\newblock URL: \url{http://projecteuclid.org/euclid.pjm/1103034558}.

\bibitem{Kawohl}
B.~Kawohl.
\newblock Symmetry results for functions yielding best constants in
  {S}obolev-type inequalities.
\newblock {\em Discrete Contin. Dynam. Systems}, 6(3):683--690, 2000.
\newblock \href {https://doi.org/10.3934/dcds.2000.6.683}
  {\path{doi:10.3934/dcds.2000.6.683}}.

\bibitem{Prashanth}
B.~Kawohl, M.~Lucia, and S.~Prashanth.
\newblock {Simplicity of the principal eigenvalue for indefinite quasilinear
  problems}.
\newblock {\em Advances in Differential Equations}, 12(4):407 -- 434, 2007.

\bibitem{Kesavan2006}
S.~Kesavan.
\newblock {\em Symmetrization \& applications}, volume~3 of {\em Series in
  Analysis}.
\newblock World Scientific Publishing Co. Pte. Ltd., Hackensack, NJ, 2006.
\newblock \href {https://doi.org/10.1142/9789812773937}
  {\path{doi:10.1142/9789812773937}}.

\bibitem{Stakgold}
E.~T. Kornhauser and I.~Stakgold.
\newblock A variational theorem for {$\nabla^2u+\lambda u=0$} and its
  applications.
\newblock {\em J. Math. Physics}, 31:45--54, 1952.
\newblock \href {https://doi.org/10.1002/sapm195231145}
  {\path{doi:10.1002/sapm195231145}}.

\bibitem{Krahn1925}
E.~Krahn.
\newblock \"{U}ber eine von {R}ayleigh formulierte {M}inimaleigenschaft des
  {K}reises.
\newblock {\em Math. Ann.}, 94(1):97--100, 1925.
\newblock \href {https://doi.org/10.1007/BF01208645}
  {\path{doi:10.1007/BF01208645}}.

\bibitem{Krahn1926}
E.~Krahn.
\newblock \"{U}ber minimaleigenschaften der {K}ugel in drei und mehr
  {D}imensionen,.
\newblock {\em Acta Comm. Univ. Tartu (Dorpat)}, A9:1--44, 1926.

\bibitem{Lieber}
G.~M. Lieberman.
\newblock Boundary regularity for solutions of degenerate elliptic equations.
\newblock {\em Nonlinear Anal.}, 12(11):1203--1219, 1988.
\newblock \href {https://doi.org/10.1016/0362-546X(88)90053-3}
  {\path{doi:10.1016/0362-546X(88)90053-3}}.

\bibitem{Makai}
E.~Makai.
\newblock On the principal frequency of a convex membrane and related problems.
\newblock {\em Czechoslovak Mathematical Journal}, 09(1):66--70, 1959.
\newblock URL: \url{http://eudml.org/doc/11967}.

\bibitem{Nazarov2000}
A.~I. Nazarov.
\newblock The one-dimensional character of an extremum point of the
  {F}riedrichs inequality in spherical and plane layers.
\newblock {\em J. Math. Sci. (New York)}, 102(5):4473--4486, 2000.
\newblock Function theory and applications.
\newblock \href {https://doi.org/10.1007/BF02672901}
  {\path{doi:10.1007/BF02672901}}.

\bibitem{Nazarov2001}
A.~I. Nazarov.
\newblock On the symmetry of extremals in the weight embedding theorem.
\newblock {\em J. Math. Sci. (New York)}, 107(3):3841--3859, 2001.
\newblock Function theory and mathematical analysis.
\newblock \href {https://doi.org/10.1023/A:1012336127123}
  {\path{doi:10.1023/A:1012336127123}}.

\bibitem{Paoli}
G.~Paoli, G.~Piscitelli, and L.~Trani.
\newblock Sharp estimates for the first {$p$}-{L}aplacian eigenvalue and for
  the {$p$}-torsional rigidity on convex sets with holes.
\newblock {\em ESAIM Control Optim. Calc. Var.}, 26:Paper No. 111, 15, 2020.
\newblock \href {https://doi.org/10.1051/cocv/2020033}
  {\path{doi:10.1051/cocv/2020033}}.

\bibitem{Payne}
L.~E. Payne and H.~F. Weinberger.
\newblock Some isoperimetric inequalities for membrane frequencies and
  torsional rigidity.
\newblock {\em J. Math. Anal. Appl.}, 2:210--216, 1961.
\newblock \href {https://doi.org/10.1016/0022-247X(61)90031-2}
  {\path{doi:10.1016/0022-247X(61)90031-2}}.

\bibitem{Polya}
G.~P\'{o}lya.
\newblock Two more inequalities between physical and geometrical quantities.
\newblock {\em J. Indian Math. Soc. (N.S.)}, 24:413--419 (1961), 1960.

\bibitem{Rayleigh}
J.~W.~S. Rayleigh, Baron.
\newblock {\em The {T}heory of {S}ound}.
\newblock Dover Publications, New York, N. Y., 1945.
\newblock 2d ed.

\bibitem{Schneider}
R.~Schneider.
\newblock {\em Convex bodies: the {B}runn-{M}inkowski theory}, volume 151 of
  {\em Encyclopedia of Mathematics and its Applications}.
\newblock Cambridge University Press, Cambridge, expanded edition, 2014.

\bibitem{Steiner}
J.~Steiner.
\newblock \"{U}ber parallele fl\"{a}chen.
\newblock {\em Monatsbericht der Akademie der Wissenschaften zu Berlin}, pages
  114--118, 1840.

\bibitem{Nagy}
B.~Sz.-Nagy.
\newblock \"{U}ber {P}arallelmengen nichtkonvexer ebener {B}ereiche.
\newblock {\em Acta Sci. Math. (Szeged)}, 20:36--47, 1959.

\bibitem{Szego}
G.~Szeg\"{o}.
\newblock Inequalities for certain eigenvalues of a membrane of given area.
\newblock {\em J. Rational Mech. Anal.}, 3:343--356, 1954.
\newblock \href {https://doi.org/10.1512/iumj.1954.3.53017}
  {\path{doi:10.1512/iumj.1954.3.53017}}.

\bibitem{Weinberger1956}
H.~F. Weinberger.
\newblock An isoperimetric inequality for the {$N$}-dimensional free membrane
  problem.
\newblock {\em J. Rational Mech. Anal.}, 5:633--636, 1956.
\newblock \href {https://doi.org/10.1512/iumj.1956.5.55021}
  {\path{doi:10.1512/iumj.1956.5.55021}}.

\end{thebibliography}
